\documentclass[twoside,leqno,10pt, A4]{amsart}
\usepackage{amsfonts}
\usepackage{amsmath}
\usepackage{amscd}
\usepackage{amssymb}
\usepackage{amsthm}
\usepackage{amsrefs}
\usepackage{latexsym}
\usepackage{mathrsfs}
\usepackage{bbm}
\usepackage{enumerate}
\usepackage{graphicx}

\usepackage{amsfonts}
\usepackage{amsmath}
\usepackage{amscd}
\usepackage{amssymb}
\usepackage{amsthm}
\usepackage{amsrefs}
\usepackage{latexsym}
\usepackage{mathrsfs}
\usepackage{bbm}
\usepackage{amscd}
\usepackage{amssymb}
\usepackage{amsthm}
\usepackage{amsrefs}
\usepackage{latexsym}
\usepackage{mathrsfs}
\usepackage{bbm}
\usepackage{enumerate}
\usepackage{graphicx}
\usepackage{color}
\begin{document}

\newtheorem{theorem}[subsection]{Theorem}
\newtheorem{proposition}[subsection]{Proposition}
\newtheorem{lemma}[subsection]{Lemma}
\newtheorem{corollary}[subsection]{Corollary}
\newtheorem{conjecture}[subsection]{Conjecture}
\newtheorem{prop}[subsection]{Proposition}
\numberwithin{equation}{section}
\newcommand{\mr}{\ensuremath{\mathbb R}}
\newcommand{\mc}{\ensuremath{\mathbb C}}
\newcommand{\dif}{\mathrm{d}}
\newcommand{\intz}{\mathbb{Z}}
\newcommand{\ratq}{\mathbb{Q}}
\newcommand{\natn}{\mathbb{N}}
\newcommand{\comc}{\mathbb{C}}
\newcommand{\rear}{\mathbb{R}}
\newcommand{\prip}{\mathbb{P}}
\newcommand{\uph}{\mathbb{H}}
\newcommand{\fief}{\mathbb{F}}
\newcommand{\majorarc}{\mathfrak{M}}
\newcommand{\minorarc}{\mathfrak{m}}
\newcommand{\sings}{\mathfrak{S}}
\newcommand{\fA}{\ensuremath{\mathfrak A}}
\newcommand{\mn}{\ensuremath{\mathbb N}}
\newcommand{\mq}{\ensuremath{\mathbb Q}}
\newcommand{\half}{\tfrac{1}{2}}
\newcommand{\f}{f\times \chi}
\newcommand{\summ}{\mathop{{\sum}^{\star}}}
\newcommand{\chiq}{\chi \bmod q}
\newcommand{\chidb}{\chi \bmod db}
\newcommand{\chid}{\chi \bmod d}
\newcommand{\sym}{\text{sym}^2}
\newcommand{\hhalf}{\tfrac{1}{2}}
\newcommand{\sumstar}{\sideset{}{^*}\sum}
\newcommand{\sumprime}{\sideset{}{'}\sum}
\newcommand{\sumprimeprime}{\sideset{}{''}\sum}
\newcommand{\sumflat}{\sideset{}{^\flat}\sum}
\newcommand{\shortmod}{\ensuremath{\negthickspace \negthickspace \negthickspace \pmod}}
\newcommand{\V}{V\left(\frac{nm}{q^2}\right)}
\newcommand{\sumi}{\mathop{{\sum}^{\dagger}}}
\newcommand{\mz}{\ensuremath{\mathbb Z}}
\newcommand{\leg}[2]{\left(\frac{#1}{#2}\right)}
\newcommand{\muK}{\mu_{\omega}}
\newcommand{\thalf}{\tfrac12}
\newcommand{\lp}{\left(}
\newcommand{\rp}{\right)}
\newcommand{\Lam}{\Lambda_{[i]}}
\newcommand{\lam}{\lambda}
\def\L{\fracwithdelims}
\def\om{\omega}
\def\pbar{\overline{\psi}}
\def\phis{\varphi^*}
\def\lam{\lambda}
\def\lbar{\overline{\lambda}}
\newcommand\Sum{\Cal S}
\def\Lam{\Lambda}
\newcommand{\sumtt}{\underset{(d,2)=1}{{\sum}^*}}
\newcommand{\sumt}{\underset{(d,2)=1}{\sum \nolimits^{*}} \widetilde w\left( \frac dX \right) }

\newcommand{\hf}{\tfrac{1}{2}}
\newcommand{\af}{\mathfrak{a}}
\newcommand{\Wf}{\mathcal{W}}

\theoremstyle{plain}
\newtheorem{conj}{Conjecture}
\newtheorem{remark}[subsection]{Remark}

\makeatletter
\def\widebreve{\mathpalette\wide@breve}
\def\wide@breve#1#2{\sbox\z@{$#1#2$}%
     \mathop{\vbox{\m@th\ialign{##\crcr
\kern0.08em\brevefill#1{0.8\wd\z@}\crcr\noalign{\nointerlineskip}%
                    $\hss#1#2\hss$\crcr}}}\limits}
\def\brevefill#1#2{$\m@th\sbox\tw@{$#1($}%
  \hss\resizebox{#2}{\wd\tw@}{\rotatebox[origin=c]{90}{\upshape(}}\hss$}
\makeatletter

\title[Upper bounds for shifted moments of Dirichlet $L$-functions to a fixed modulus over function fields]{Upper bounds for shifted moments of Dirichlet $L$-functions to a fixed modulus over function fields}

\author[S. Baier] {Stephan Baier}
\address{Stephan Baier\\
	Ramakrishna Mission Vivekananda Educational and Research Institute\\
	Department of Mathematics\\
	G.\ T.\ Road, PO~Belur Math, Howrah, West Bengal~711202\\
	India}
\email{stephanbaier2017@gmail.com}

\author[P. Gao]{Peng Gao}
\address{School of Mathematical Sciences, Beihang University, Beijing 100191, China}
\email{penggao@buaa.edu.cn}

\begin{abstract}
 In this paper, we establish sharp upper bounds on shifted moments of the family of Dirichlet $L$-functions to a fixed modulus over function fields. We apply the result to obtain upper bounds on moments of Dirichlet character sums over function fields. 
\end{abstract}

\maketitle



\section{Introduction}
\label{sec 1}

  It is an important subject to study the moments of families of $L$-functions in number theory as they have many significant applications. Although it is a challenge to evaluate higher moments in general, there are now many conjectured asymptotic formulas as well as sharp lower and upper bounds concerning the moments. For example, a remarkable method of K. Soundararajan \cite{Sound2009} as well as its refinement by A. J. Harper  \cite{Harper} can be applied to establish sharp upper bounds for these moments under the assumption of the generalized Riemann hypothesis (GRH).
  
  The family of Dirichlet $L$-functions to a fixed modulus is a typical family of interest with its moments being extensively investigated in the literature, especially in the number fields setting.  In \cite{Munsch17}, M. Munsch adapted the method in \cite{Sound2009} to show that under GRH, for a large fixed modulus $q$, any  positive integer $2k$, and real numbers $t_j$ such that $t_j \ll \log q$, 
\begin{equation}
  \label{munsch1}
\sum_{\chi\in X_q^*}|L(1/2+it_1,\chi) \cdot L(1/2+it_2, \chi)\cdots L(1/2+it_{2k},\chi)|\ll_{\epsilon, k} \varphi(q) (\log q)^{k/2+\varepsilon} \prod_{1\leq j<l\leq 2k} \tilde g^{1/2}(|t_j-t_l|), 
\end{equation}
 where $X_q^*$ denotes the set of primitive Dirichlet characters modulo $q$, $\varphi$ is the Euler totient function, and where $\tilde g:\mathbb{R}_{\geq 0} \rightarrow \mathbb{R}$ is the function defined by
$$\tilde g(x) =\begin{cases}
\min (\frac{1}{x}, \log q)  & \text{if } x\leq \frac{1}{100}, \\
\log \log q & \text{if } x\geq \frac{1}{100}.
\end{cases}
$$
  Here we point out that one may regard $\tilde g$ as a correlation factor that measures how the values of the $L$-functions correlate to each other at various points on the critical line.  In \cite{Szab}, B. Szab\'o applied the method \cite{Harper} to improve the estimate given in \eqref{munsch1} by removing the $(\log q)^{\varepsilon}$ factor, extending it to hold for $t_j \ll q^C$ for any fixed real number $C$ and replacing $\tilde g$ by a function $g(x)$ given by
$$g(x) =\begin{cases}
\log q  & \text{if } x\leq \frac{1}{\log q} \text{ or } x\geq e^q, \\
\frac{1}{x} & \text{if }   \frac{1}{\log q}< x< 10, \\
\log \log x & \text{if } 10\leq x\leq e^q.
\end{cases}
$$
  The same type of upper bound was obtained by M. J. Curran  \cite{Curran} for shifted moments of the Riemann zeta function, using similar approaches. We point out here that the correlation factor in the work of Szab\'o  can be taken (see \cite[Lemma 2]{Szab}) to be $g(t)=|\zeta(1+1/\log q+it)|$, where $\zeta(s)$ is the Riemann zeta function. This is consistent with the correlation factor appearing in the work of Curran. 
 
  Compared to values of $L$-functions on arbitrary points on the critical line, much more attention has been paid to central values of $L$-functions in the literature, due to rich arithmetic meanings they carry. The works of Szab\'o and Curran imply that it is also meaningful to study values of $L$-functions off the central point, as they may be used to determine the symmetry type of the corresponding family of $L$-functions. Moreover, the shifted moments can be applied to achieve bounds for moments of character sums, as done in \cites{Szab, G&Zhao2024, Gao2024}.
 
  It is the aim of this paper is to extend the above mentioned result of Szab\'o to the function fields setting and then apply it to obtain similar bounds for moments of character sums.  To state our results, we denote by $A=\mathbb{F}_{q}[T]$ the polynomial ring over $\mathbb{F}_{q}$, where $\mathbb{F}_{q}$ is a fixed finite field of cardinality $q$. We denote by $d(f)$ the degree of any $f \in A$, and define the norm $|f|$ to be $|f|=q^{d(f)}$ for $f\neq 0$ and $|f|=0$ for $f=0$.  Let $\chi$ be a Dirichlet character modulo $Q$ defined in Section~\ref{sec 2} for a fixed polynomial $Q \in A$ of degree larger than $1$ and let $L(s,\chi)$ be the Dirichlet $L$-function associated to $\chi$. Denote by $X_Q^*$ the set of primitive characters modulo $Q$ and $\varphi(Q)$ for the Euler totient function on $A$.  Our first result concerns upper bounds for shifted moments of the family of these Dirichlet $L$-functions, which is analogue to \cite[Theorem 1]{Szab}.
\begin{theorem}
\label{t1}
Keep the notations above. Let $k\geq 1$ be a fixed integer and $a_1,\ldots, a_{k}$ be fixed positive real numbers. Then
for any real $k$-tuple $t=(t_1,\ldots ,t_{k})$, we have
$$\sum_{\chi \in X_Q^*} \big| L\big(1/2+it_1,\chi \big) \big|^{a_1} \cdots \big| L\big(1/2+it_{k},\chi \big) \big|^{a_{k}}\ll \varphi(Q)(\log |Q|)^{(a_1^2+\cdots +a_{k}^2)/4} \prod_{1\leq j<l\leq k}  \big|\zeta_A(1+i(t_j-t_l)+\frac 1{\log |Q|}) \big|^{a_ja_l/2},  $$
 where $L(s, \chi)$ is the $L$-function associated to $\chi$ defined in \eqref{Ldef} and where $\zeta_A(s)$ is the zeta function associated to $A$ defined in \eqref{zetadef}. 

  Consequently, we have
$$\sum_{\chi \in X_Q^*} \big| L\big(1/2+it_1,\chi \big) \big|^{a_1} \cdots \big| L\big(1/2+it_{k},\chi \big) \big|^{a_{k}}\ll \varphi(Q)(\log |Q|)^{(a_1^2+\cdots +a_{k}^2)/4} \prod_{1\leq j<l\leq k} (\min (\log |Q|, \frac {1}{\overline{\log q |t_i-t_j|}}))^{a_ja_l/2},  $$
   where we define $\overline {\theta}=\min_{n \in \mz} |\theta-2\pi n|$ for any real $\theta$.

   The implied constants in the above estimations depend on $k$, $a_j$ but not on $Q$ or the $t_j$.
\end{theorem}

   The proof of Theorem \ref{t1} follows closely the proof of \cite[Theorem 1]{Szab}. The special case of $t_j=0$ and $Q$ being prime in Theorem \ref{t1} has been established in \cite[Proposition 1.3]{G&Zhao2024-2}.  We remark here that Theorem \ref{t1} is valid unconditionally as the Riemann hypothesis is true for the function fields case. Moreover, Theorem \ref{t1} holds without any restrictions on $t_j$. In fact,  our result depends on $\overline{\log q |t_i-t_j|}$ instead of $|t_i-t_j|$ since the $L$-functions $L(1/2+it, \chi)$ in the function fields setting are periodic functions of $t$.

  As shown in Section \ref{sec 2.1}, we have $L(s, \chi)=\mathcal{L}(q^{-s},\chi)$, where $\mathcal{L}(q^{-s},\chi)$ is defined in \eqref{Ludef}. We now set $u=q^{-1/2}e^{i\theta}$ to see that
$$ \mathcal{L}(\frac {e^{i\theta}}{\sqrt{q}},\chi) = L\big(\frac 12- \frac {\theta i}{\log q}, \chi \big ).$$

   We apply the above to obtain a restatement of Theorem \ref{t1} in terms of $\mathcal{L}(u,\chi)$.
\begin{corollary}
\label{cor1}
With the notation as in the statement of Theorem \ref{t1}, we have
\begin{align}
\label{mathcalLestimation}
 \sum_{\chi \in X_Q^*} \big| \mathcal{L}(\frac {e^{i\theta_1}}{\sqrt{q}},\chi)\big|^{a_1} \cdots \big| \mathcal{L}(\frac {e^{i\theta_{k}}}{\sqrt{q}},\chi) \big|^{a_{k}}\ll \varphi(Q)(\log |Q|)^{(a_1^2+\cdots +a_{k}^2)/4} \prod_{1\leq j<l\leq k} ( \min (\log |Q|, \frac {1}{\overline{|\theta_j-\theta_l|}}))^{a_ja_l/2}.
\end{align}
 Here the implied constant depends on $k$, $a_j$ but not on $Q$ or the $\theta_j$.
\end{corollary}  
  
  As an application of Theorem \ref{t1}, we next consider estimations for moments of Dirichlet character sums. We define for any real $m \geq 0$, 
\begin{align*}
  S_m(Q, Y)= \sum_{\chi\in X_Q^*}\bigg|\sum_{|f|\leq Y} \chi(f)\bigg|^{2m}. 
\end{align*}  
   Here and throughout the paper, we adapt the convention that when considering a sum over some subset $S$ of $A$, the symbol $\sum_{f \in S}$ stands for a sum over monic $f \in S$, unless otherwise specified.
  
   We apply arguments similar to those used in the proof of \cite[Theorem 3]{Szab} to establish the following bounds on $S_m(Q, Y)$. 
\begin{theorem}
\label{quadraticmean}
With the notation as above, for any real number $m>2$ and large $Y$, we have
\begin{align*}
 S_m(Q,Y)  \ll \varphi(Q)Y^m(\log |Q|)^{(m-1)^2}.
\end{align*}
\end{theorem}

\section{Preliminaries}
\label{sec 2}

\subsection{Backgrounds on function fields}
\label{sec 2.1}

We recall some basic facts concerning function fields in this section, most of which can be found in \cite{Rosen02}.   
 The zeta function $\zeta_A(s)$ associated to $A=\mathbb{F}_{q}[T]$ for $\Re(s)>1$ is defined by
\begin{equation}
\label{zetadef}
\zeta_A(s)=\sum_{\substack{f\in A}}\frac{1}{|f|^{s}}=\prod_{P}(1-|P|^{-s})^{-1},
\end{equation}
  where we make the convention that we reserve the symbol $P$ for a monic, irreducible polynomial in $A$ throughout the paper and we refer to $P$ as a prime in $A$.  As there are $q^n$ monic polynomials of degree
$n$, we deduce that
\begin{equation}
\label{zetaAdef}
\zeta_A(s)=\frac{1}{1-q^{1-s}}.
\end{equation}
 This defines $\zeta_A(s)$ on the entire complex plane with a simple pole at $s = 1$.

   A Dirichlet character $\chi$ modulo $f \in A$ is defined in a similar way to that in the number fields case.  More precisely, such $\chi$ is a homomorphism from $(A/fA)^*$ to $\mc$ so that $\chi(\overline{g})=0$ for any $(g, f) \neq1$, where $\overline{g}$ is the coset to which $g$ belongs in $A/fA$.  Throughout the paper, we shall regard $\chi$ as a function defined on $A$ instead of $(A/fA)^*$ such that we have $\chi(g) = \chi(\overline{g})$ for any $g \in A$. For a fixed modulus $f \in A$, we denote by $\chi_0$ the principal character modulo $f$ so that $\chi_0(g)=1$ for any $(g, f)=1$. A character $\chi$ modulo $f$ is said to be primitive if it cannot be factored through $(A/f'A)^*$ for any proper divisor $f'$ of $f$.  The $L$-function associated to $\chi$ for $\Re(s)>1$ is defined to be
\begin{equation}
\label{Ldef}
L(s,\chi)=\sum_{\substack{f\in A}}\frac{\chi(f)}{|f|^{s}}=\prod_{P}(1-\chi(P)|P|^{-s})^{-1}.
\end{equation}

  We often write $L(s,\chi)=\mathcal{L}(u,\chi)$ via a change of variables $u=q^{-s}$, where
\begin{align}
\label{Ludef}
  \mathcal{L}(u,\chi) = \sum_{f \in A} \chi(f) u^{d(f)} = \prod_P (1-\chi(P) u^{d(P)})^{-1}.
\end{align}

\subsection{Sums over primes}

In this section we include some estimations concerning sums over primes in the function fields setting.  The first one reproduces \cite[Lemma 2.2]{G&Zhao12}.
\begin{lemma}
\label{RS}
  Denote by $\pi(n)$ the number of primes of degree $n$. We have
\begin{equation}
\label{ppt}
\pi(n) = \frac{q^n}{n}+O \Big(  \frac{q^{n/2}}{n} \Big).
\end{equation}
   For $x \geq 2$ and some constant $b$, we have
\begin{align}
\label{logp}
\sum_{|P| \le x} \frac {\log |P|}{|P|} =& \log x + O(1) \quad \mbox{and} \\
\label{lam2p}
\sum_{|P| \le x} \frac{1}{|P|} =& \log \log x + b+ O\left( \frac {1}{\log x} \right).
\end{align}
\end{lemma}

  Our next result can be regarded as a generalization of \eqref{lam2p}, which is an analogue to \cite[Lemma 3.2]{Kou}, \cite[Lemma 2.6]{Curran} and \cite[Lemma 2]{Szab} in the number fields setting. 
\begin{lemma}
\label{mertenstype}
Let $\alpha>0$, then for $x \geq 2$,
\begin{align}
\label{sumcosp}
\begin{split}
\sum_{|P|\leq x} \frac{\cos(\alpha \log |P|) }{|P|}=& \log |\zeta_A(1+1/\log x +i\alpha)|+O(1) \\
  =& \log \min (\log x, \frac {1}{\overline {\alpha \log q}})+O(1).
\end{split}
\end{align}
\end{lemma}
\begin{proof}
  The first equality in \eqref{sumcosp} is an analogue to the special case given in \cite[Lemma 3.2]{Kou} concerning the number fields case. We note that the arguments given in the proof of \cite[Lemma 3.2]{Kou} carry over to the function fields setting as well, since the zeta function in our case also has an Euler product. Also, applying \eqref{logp}, \eqref{lam2p} and partial summation implies that
\begin{align}
\label{ppowerest}
 \sum_{\substack{|P| \leq x }}\big (\frac 1{|P|}-\frac {1}{|P|^{1+1/\log x}}\big )+\sum_{\substack{|P|> x }}\frac {1}{|P|^{1+1/\log x}}=O(1).
\end{align}
  This allows us to establish the first equality in \eqref{sumcosp}. For the second equality in \eqref{sumcosp}, we apply \eqref{ppt} to see that
\begin{align}
\label{sumcospF}
\sum_{|P|\leq x} \frac{\cos(\alpha \log |P|) }{|P|}= \sum_{d(P)\leq \log_q x} \frac{\cos(\alpha \log |P|) }{|P|}=\sum_{n\leq \log_q x}
\frac{\cos(\alpha \log q^n) }{q^n}\cdot \pi(n)=F(\log_q x,  \alpha \log q)+O(1),
\end{align}
   where we define for real numbers $h, \theta$ with $h \geq 1$,
\begin{align*}
F(h, \theta)=\sum^h_{n=1}\frac {\cos (n\theta)}{n}.
\end{align*}
  Note that it follows from the discussions given on \cite[p. 11]{B&F20} that we have
\begin{align*}
F(h, \theta)=\log \min (h, \frac {1}{\overline {\theta}})+O(1).
\end{align*}
  Applying the above into \eqref{sumcospF}, we readily deduce the second equality in \eqref{sumcosp}. 

  Another way to establish the second equality in \eqref{sumcosp} is pointed out to us by the referee. We apply \eqref{zetaAdef} to see that 
\begin{align}
\label{zetaAexp}
\begin{split}
\zeta_A(1+1/\log x +i\alpha)=\frac {1}{1-q^{-1/\log x-i\alpha}}=(1-e^{z})^{-1},
\end{split}
\end{align}
  where we set $z=-\log q/\log x-i\alpha \log q$. We may assume that $\alpha \log q \in [-\pi, \pi]$ and we may further assume that $x$ is large enough, say $x \geq q^2$ for otherwise the second equality in \eqref{sumcosp} holds trivially. It then follows that $|z| \leq \pi+1/2$. We then apply \cite[(B.8)]{MVa1} to see that we have
\begin{align}
\label{ezexpansion}
\begin{split}
(1-e^{z})^{-1}=\frac 1z+\sum^{\infty}_{k=1}\frac {B_k}{k!}z^{k-1},
\end{split}
\end{align}
 where $B_k$ is the Bernoulli number satisfying $B_1=-1/2, B_k=0$ for odd $k \geq 3$, and that for even $k \geq 2$ (see \cite[(B.20)]{MVa1}), 
\begin{align}
\label{bkbound}
\begin{split}
 |B_k|<2(k!)(2\pi)^{-k}(1+2^{2-k}). 
\end{split}
\end{align}
  It follows from \eqref{ezexpansion} and \eqref{bkbound} that 
\begin{align}
\label{ezexpansionbound}
\begin{split}
(1-e^{z})^{-1}=\frac 1z+O(1). 
\end{split}
\end{align}
  The second equality in \eqref{sumcosp} now follows from \eqref{zetaAexp} and \eqref{ezexpansionbound}. This completes the proof of the lemma.
\end{proof}

\subsection{Perron’s formula}

 As an easy consequence of Cauchy's residue theorem, we have the following analogue of Perron’s formula in function fields (see \cite[(2.6)]{Florea17-2}).
\begin{lemma}
\label{lem4}
 Suppose that the power series $\sum^{\infty}_{n=0} a(n)u^n$ is absolutely convergent in $|u| \leq  r < 1$, then for any integer $N \geq 0$,
\begin{align}
\label{perron1}
\begin{split}
 \sum_{n \leq N}a(n)=& \frac 1{2\pi i}\oint_{|u|=r} \Big ( \sum^{\infty}_{n=0}a(n)u^n\Big )\frac {du}{(1-u)u^{N+1}}.
\end{split}
\end{align}
\end{lemma}

\section{Proof of Theorem \ref{t1}}
\label{sec 2'}

    The proof of Theorem \ref{t1} proceeds along the lines of the proof of \cite[Theorem 1]{Szab}.  We begin with some preliminary bounds for $L$-functions.
\subsection{Various bounds for $L$-functions}

  For any fixed $\chi \neq \chi_0$, it follows from \cite[Proposition 4.3]{Rosen02} that the function $L(s,\chi)$ is a polynomial in $q^{-s}$ of degree $m$ that does not exceed $d(Q)-1$, where $d(Q)$ is the degree of the modulus $Q$ of $\chi$.  It follows from this and the Riemann hypothesis for the function fields, proven by A. Weil \cite{Weil}, that there are complex numbers $\alpha_j, 1\leq j \leq m$ satisfying $|\alpha_j| = 1$ such that
\begin{align*}
 L(s, \chi) =\prod^m_{j=1}(1-\alpha_jq^{1/2-s}).
\end{align*}
  
  We also deduce from \eqref{Ldef} that
\begin{align*}
 -\frac {L'(s, \chi)}{L(s, \chi)} =\sum_{P}\sum_{j\geq 1}\frac {\chi^j(P)(\log|P|)}{|P|^{js}}=:\sum_{n \geq 0}\frac {\lambda_{\chi}(n)}{q^{ns}}, 
\end{align*}
  where
\begin{align}
\label{lambdadef}
 \lambda_{\chi}(n)=\log q\sum_{j|n}\sum_{\substack{P \\d(P^j)=n}}d(P)\chi^j(P).
\end{align}

  We then proceed as in the proof of \cite[Proposition 4.3]{BFK} with the above notation to see that upon setting $z=it, \sigma_0=1/2+1/\log q^h$, the last display on \cite[p. 1875]{BFK} implies that for any integer $h \geq 0$, 
\begin{align}
\label{logLupperboundgen}
\log \big| L(\tfrac{1}{2}+it, \chi) \big| \leq \frac{m}{h} + \frac{1}{\log q^h} \Re \bigg(  \sum_{\substack{n \leq h}} \frac{
\lambda_{\chi}(n) \log q^{h - n}}{q^{n\big( 1/2+it+1/(h \log q) \big)} \log q^n} \bigg).
\end{align}
 As the right-hand side expression above is an increasing function of $m$, we may set $m=d(Q)-1$ to deduce from \eqref{lambdadef} and \eqref{logLupperboundgen}
the following analogue of \cite[Proposition 4.3]{BFK}.
\begin{proposition}
\label{prop-ub}
 Let $\chi$ be a non-principal primitive character modulo $Q$ and let $m=d(Q)-1$.  We have for $h \leq m$, $t \in \mr$,
\begin{align}
\label{logLupperbound}
\log \big| L(\tfrac{1}{2}+it, \chi) \big| \leq \frac{m}{h} + \frac{1}{h} \Re \bigg(  \sum_{\substack{j \geq 1 \\ d(P^j) \leq h}} \frac{
\chi(P^j) \log q^{h - j \deg(P)}}{|P|^{j\big( 1/2+it+1/(h \log q) \big)} \log q^j} \bigg).
\end{align}
\end{proposition}

  Observe further that Lemma \ref{RS} implies that the terms on the right-hand side of \eqref{logLupperbound} corresponding to $P^j$ with $j \geq 3$
    contribute $O(1)$. This allows us to deduce from \eqref{logLupperbound} by setting $|Q|=q^{d(Q)}, x=q^h$ there to see that
\begin{align}
\label{logLbound}
\begin{split}
 \log  |L(\half+it, \chi)| \leq & \Re \left( \sum_{\substack{  |P| \leq x }} \frac{\chi (P)}{|P|^{1/2+it+1/\log x}}
 \frac{\log (x/|P|)}{\log x} +
 \sum_{\substack{  |P| \leq x^{1/2} }} \frac{\chi (P^2)}{|P|^{1+2it+2/\log x}}  \frac{\log (x/|P|^2)}{2\log x} \right)
 +\frac{\log |Q|}{\log x} + O(1) \\
 \leq  & \Re \left( \sum_{\substack{  |P| \leq x }} \frac{\chi (P)}{|P|^{1/2+it+1/\log x}}
 \frac{\log (x/|P|)}{\log x} + \frac 12
 \sum_{\substack{  |P| \leq x^{1/2} }} \frac{\chi (P^2)}{|P|^{1+2it+2/\log x}} \right)
 +\frac{\log |Q|}{\log x} + O(1),
\end{split}
 \end{align}
   where the last estimation above follows from \eqref{logp}.

   We apply \eqref{logp} again to see that
\begin{align}
\label{sumprimesquaresimplfied}
\begin{split}
 \sum_{\substack{  |P| \leq x^{1/2} }} \Big ( \frac{\chi (P^2)}{|P|^{1+2it}}-\frac{\chi (P^2)}{|P|^{1+2it+2/\log x}}  \Big ) \ll \sum_{\substack{  |P| \leq x^{1/2} }}\frac {\log |P|}{|P|\log x} = O(1).
\end{split}
 \end{align}

   We derive from \eqref{logLbound} and \eqref{sumprimesquaresimplfied} that
\begin{align}
\label{logLboundsimplified}
\begin{split}
 \log  |L(\half+it, \chi)| \leq & \Re \left( \sum_{\substack{  |P| \leq x }} \frac{\chi (P)}{|P|^{1/2+it+1/\log x}}
 \frac{\log (x/|P|)}{\log x} + \frac 12
 \sum_{\substack{  |P| \leq x^{1/2} }} \frac{\chi (P^2)}{|P|^{1+2it}} \right)
 +\frac{\log |Q|}{\log x} + O(1).
\end{split}
 \end{align}

   We set $x=\log |Q|$ in the above expression and estimation everything trivially to see that for any non-principal primitive character $\chi$ modulo $Q$, we have  for some constant $C$,
\begin{align}
\label{singleLbound}
\begin{split}
  |L(\half+it, \chi)| \ll & \exp\big(C\frac{\log |Q|}{\log \log |Q|}\big).
\end{split}
 \end{align}

   Moreover, we deduce readily from \eqref{logLboundsimplified} the following upper bound for sums involving $\log |L(1/2+it, \chi)|$ with various $t$, which is analogous to \cite[Proposition 2]{Szab}.
\begin{proposition}
\label{prop1}
Let $k$ be a positive integer and let $a_1,a_2,\ldots, a_{k}$ be positive constants, $x\geq 2$. Let $a:=a_1+\cdots+ a_{k}+10$. Let $Q$ be a large modulus and let $t_1,\ldots, t_{k}$ real numbers. For any monic polynomial $f$, let
$$h(f)=:\frac{1}{2}(a_1|f|^{-it_1}+\cdots +a_{k}|f|^{-it_{k}}).$$
Then
\begin{align*}
\begin{split}
 & \sum^{k}_{j=1}a_j\log |L(1/2+it_j,\chi)| \\
    \leq & 2\cdot\Re \sum_{|P|\leq x} \frac{h(P)\chi(P)}{|P|^{1/2+1/\log x}}\frac{\log (x/|P|)}{\log x}+\Re\sum_{|P|\leq x^{1/2}} \frac{h(P^2)\chi(P^2)}{|P|}+a\frac{\log |Q|}{\log x}+O(1).
\end{split}
\end{align*}
\end{proposition}

   Lastly, we apply Proposition \ref{prop1} and argue as in the proof of \cite[Proposition 1]{Szab} to arrive at the following crude estimation on shifted moments 
   of Dirichlet $L$-functions. 
\begin{proposition}
\label{crude_prop}
    With the notation as in Proposition \ref{prop1},  we have
\begin{equation*}
    \sum_{\chi \in X_Q^*} \big| L\big(1/2+it_1,\chi \big) \big|^{a_1} \cdots \big| L\big(1/2+it_{k},\chi \big)\big|^{a_{k}}\leq \varphi(Q)(\log |Q|)^{O(1)}.
\end{equation*}
Here the implied constant depends on $k, a_j$, but not on $Q$ or the $t_j$.
\end{proposition}

\subsection{Completion of the proof}
   We follow the treatment in \cite[Section 4]{Szab} to write $\beta_0=0$, $\beta_i=\frac{20^{i-1}}{(\log\log |Q|)^2}$ for $i\geq 1$ and let $\mathcal{I}=1+\max \{i: \beta_i\leq e^{-10000a^2}\}$. For any $1\leq i\leq j\leq \mathcal{I}$, we define
$$G_{(i,j)}(\chi)=\sum_{|Q|^{\beta_{i-1}}<|P|\leq |Q|^{\beta_i}} \frac{\chi(P)h(P)}{|P|^{1/2+1/\beta_{j}\log |Q| } } \frac{\log( |Q|^{\beta_j }/|P| )}{\log( |Q|^{\beta_j})}.$$
  We also define
$$\mathcal{T}=\{ \chi \in X_Q^* : |\Re G_{(i,\mathcal{I})}( \chi )|\leq \beta_i^{-3/4}, \forall 1\leq i\leq \mathcal{I} \},$$
and we set for each $0\leq j<\mathcal{I}$,
$$\mathcal{S}(j)=\{\chi \in X_Q^*:|\Re G_{(i,l)}(\chi)|\leq \beta_{i}^{-3/4} \; \forall 1\leq i\leq j, \, \forall i\leq l\leq \mathcal{I}\, \text{but}\, |\Re G_{(j+1,l)}(\chi)|>\beta_{j+1}^{-3/4} \, \text{for some}\, j+1\leq l\leq \mathcal{I} \}.$$

   It follows from the arguments in \cite[Section 4]{Szab} that in order to prove Theorem \ref{t1}, it suffices to establish Propositions \ref{prop1}-\ref{crude_prop} together with the following three lemmas.
\begin{lemma}
\label{mainfirst}
We have
\begin{align*}
\begin{split}
 &\sum_{\chi \in \mathcal{T}}\exp^2\bigg(  \Re \sum_{|P|\leq |Q|^{\beta_{\mathcal{I}}}}  \frac{\chi(P)h(P)}{|P|^{1/2+1/(\beta_{\mathcal{I}}\log |Q| ) }} \frac{\log( |Q|^{\beta_{\mathcal{I} } }/|P| )}{\log( |Q|^{\beta_{\mathcal{I}}})}\bigg) \\
\ll & \varphi(Q)(\log |Q|)^{(a_1^2+\cdots +a_{k}^2)/4} \prod_{1\leq j<l\leq k} \Big |\zeta_A(1+i(t_j-t_l)+\frac 1{\log |Q|})\Big |^{a_ja_l/2}.
\end{split}
\end{align*}
\end{lemma}

\begin{lemma}
\label{mainsecond}
We have $|\mathcal{S}(0)|\ll |Q|e^{-(\log\log |Q|)^2}$ and for $1\leq j\leq \mathcal{I}-1$ we have
\begin{align*}
\begin{split}
 &   \sum_{\chi \in \mathcal{S}(j)}\exp^2\bigg( \Re \sum_{|P|\leq |Q|^{\beta_j}} \frac{\chi(P)h(P)}{|P|^{1/2+1/(\beta_j\log |Q|)}}\frac{\log(|Q|^{\beta_j}/|P|)}{\log(|Q|^{\beta_j}) }\bigg) \\
 \ll &e^{-\beta_{j+1}^{-1} \log (\beta_{j+1}^{-1})/200 } \varphi(Q)(\log |Q|)^{(a_1^2+\cdots +a_{k}^2)/4} \prod_{1\leq j<l\leq k}\Big | \zeta_A(1+i(t_j-t_l)+\frac 1{\log |Q|})\Big |^{a_ja_l/2}.
\end{split}
\end{align*}
\end{lemma}

\begin{lemma}
\label{important}
The statements of the previous two lemmas remain true if we replace the Dirichlet polynomial by
$$\Re \sum_{|P|\leq |Q|^{\beta_j}} \frac{\chi(P)h(P)}{|P|^{1/2+1/(\beta_j\log |Q|)}}\frac{\log(|Q|^{\beta_j}/|P|)}{\log(|Q|^{\beta_j}) }+\frac{1}{2}\Re \sum_{|P|\leq |Q|^{\beta_j /2}}\frac{\chi(P^2)h(P^2) }{|P|}.$$
\end{lemma}

  Lemmas \ref{mainfirst}--\ref{important} can be proved the same way as Lemma 4--6 in \cite{Szab} and these proofs rely crucially on the facts that unique prime factorization applies for function fields and one has the same type of orthogonality relation for Dirichlet characters over function fields, that is, for monic $f, \ g \in A$:
\begin{align}
\label{orthrel}
  \sum_{\substack{ \chi \shortmod Q }} \chi(f) \overline{\chi}(g)=\left\{
\begin{array}{ll}
  \varphi(Q), & f \equiv g \pmod Q, \\ \\
   0, & \text{otherwise}.
\end{array}
\right.
\end{align}

  The proofs of Lemma \ref{mainfirst} and \ref{mainsecond} follow by adapting the proofs for Lemma 4 and 5 in \cite{Szab} for the function fields setting in a straightforward way, upon using Lemma \ref{mertenstype}. The proof of Lemma \ref{important} also follows from that given for Lemma 6 in \cite{Szab}, upon using \eqref{singleLbound} to see that one may ignore the contribution from quadratic characters. Moreover, we apply \eqref{ppowerest} and apply partial summation to see that for any non-quadratic character $\chi$,
\begin{align*}
\begin{split}
 &   \sum_{\log |Q| < |P| \leq |Q| }\frac {\chi(P^2)h(P^2)}{|P|} \ll 1.
\end{split}
\end{align*}
  Thus we may again truncate the Dirichlet polynomial coming from the squares of primes at $\log |Q|$, as in the proof of Lemma 6 in \cite{Szab}. The arguments in the proof of Lemma 6 in \cite{Szab} then lead to the proof of Lemma \ref{important}. This then completes the proof of Theorem \ref{t1}. 

\section{Proof of Theorem \ref{quadraticmean}}
\label{sec 3}

 Our proof of Theorem \ref{quadraticmean} uses ideas in the proof of \cite[Theorem 3]{Szab}. Without loss of generality, we may assume that $\log_q Y$ is a positive integer. We first apply Perron’s formula given in \eqref{perron1} to see that for a small $r>0$,
\begin{align}
\label{sumchi}
\sum_{|f| \leq  Y} \chi(f)=& \frac 1{2\pi i}\oint_{|u|=r} \Big ( \sum_{f} \chi(f) u^{d(f)}\Big )\frac {du}{(1-u)u^{\log_q Y+1}}=\frac 1{2\pi i}\oint_{|u|=r} \frac {\mathcal{L}(u,\chi) du}{(1-u)u^{\log_q Y+1}}.
\end{align}

  Choosing $r= q^{-1/2}$, we obtain that
\begin{align*}
   \sum_{|f| \leq  Y} \chi(f)=\frac 1{2\pi i}\oint_{|u|=q^{-1/2}} \frac {\mathcal{L}(u,\chi) du}{(1-u)u^{\log_q Y+1}}= \frac {Y^{1/2}}{2\pi i}\oint_{|u|=1} \frac {\mathcal{L}(u/\sqrt{q},\chi) du}{(1-u/\sqrt{q})u^{\log_q Y+1}}.
\end{align*}
    We deduce from the above that
\begin{align*}
  \sum_{\chi \in X_Q^*} \Big |\sum_{|f| \leq  Y} \chi(f)\Big |^{2m} \ll & Y^m\sum_{\chi \in X_Q^*} \Big |\oint_{|u|=1} \frac {\mathcal{L}(u/\sqrt{q},\chi) du}{(1-u/\sqrt{q})u^{\log_q Y+1}}\Big |^{2m} \ll Y^m\sum_{\chi \in X_Q^*} \Big (\int_{0}^{2\pi} \Big |\mathcal{L}(e^{it}/\sqrt{q},\chi) \Big |dt \Big )^{2m}.
\end{align*}

  It follows from the above that in order to prove Theorem \ref{quadraticmean}, it suffices to establish the following result.
\begin{proposition}
\label{t3prop}
 With the notation as above, we have for any real number $m > 2$,
\begin{align}
\label{finiteintest}
\begin{split}
 & \sum_{\chi \in X_Q^*} \Big (\int_{0}^{2\pi} \Big |\mathcal{L}(e^{it}/\sqrt{q},\chi)  \Big |dt\Big )^{2m} \ll  \varphi(Q)(\log |Q|)^{(m-1)^2}.
\end{split}
\end{align}
\end{proposition}
\begin{proof}
  Note that we have $|\mathcal{L}(e^{it}/\sqrt{q},\chi)\Big |=|\overline{\mathcal{L}(e^{i(2\pi-t)}/\sqrt{q},\overline \chi)}\Big |$. Also, when $\chi \in X_Q^*$, so is $\overline \chi \in X_Q^*$. It follows that
\begin{align}
\label{finiteintest0}
\begin{split}
 &  \sum_{\chi \in X_Q^*} \Big (\int_{0}^{2\pi} \Big |\mathcal{L}(e^{it}/\sqrt{q},\chi)  \Big |dt\Big )^{2m} \ll \sum_{\chi \in X_Q^*}\Big ( \int_{0}^{\pi} \Big |\mathcal{L}(e^{it}/\sqrt{q},\chi)  \Big |dt\Big )^{2m}.
\end{split}
\end{align}

  Moreover, we see that
\begin{align}
\label{finiteintest1}
\begin{split}
  \sum_{\chi \in X_Q^*}\Big ( \int_{0}^{\pi} \Big |\mathcal{L}(e^{it}/\sqrt{q},\chi)  \Big |dt\Big )^{2m} \ll & \sum_{\chi \in X_Q^*}\Big ( \int_{0}^{\pi/2} \Big |\mathcal{L}(e^{it}/\sqrt{q},\chi)  \Big |dt\Big )^{2m}+\sum_{\chi \in X_Q^*}\Big ( \int_{\pi/2}^{\pi} \Big |\mathcal{L}(e^{it}/\sqrt{q},\chi)  \Big |dt\Big )^{2m} \\
 =& \sum_{\chi \in X_Q^*}\Big ( \int_{0}^{\pi/2} \Big |\mathcal{L}(e^{it}/\sqrt{q},\chi)  \Big |dt\Big )^{2m}+\sum_{\chi \in X_Q^*}\Big ( \int_{\pi/2}^{0} \Big |\mathcal{L}(e^{i(\pi-t)}/\sqrt{q},\chi) \Big |d(\pi-t) \Big )^{2m} \\
 =& \sum_{\chi \in X_Q^*}\Big ( \int_{0}^{\pi/2} \Big |\mathcal{L}(e^{it}/\sqrt{q},\chi) \Big |dt \Big )^{2m}+\sum_{\chi \in X_Q^*}\Big ( \int_{0}^{\pi/2} \Big |\mathcal{L}(-e^{-it}/\sqrt{q},\chi)  \Big |dt\Big )^{2m}.
\end{split}
\end{align}

   We treat the first sum in the last expression above by deducing via symmetry that,
\begin{align}
\label{Lintdecomp}
    \Big ( \int_{0}^{\pi/2} \Big |\mathcal{L}(e^{it}/\sqrt{q},\chi)  \Big |dt\Big )^{2m}
      \ll \int_{[0,\pi/2]^k}\prod_{a=1}^k|\mathcal{L}(e^{it_a}/\sqrt{q},\chi)| \cdot \bigg(\int_{\mathcal{D} }|\mathcal{L}(e^{iv}/\sqrt{q},\chi)|dv \bigg)^{2m-k} d\mathbf{t},
\end{align}
where $k <2m$ is a positive integer and $\mathcal{D}=\mathcal{D}(t_1,\ldots,t_k)=\{ v\in [0,\pi/2]:|t_1-v|\leq |t_2-v|\leq \ldots \leq |t_k-v| \}$.

  We let $\mathcal{B}_1=\big[-\frac{\pi}{2\log |Q|},\frac{\pi}{2\log |Q|}\big]$ and $\mathcal{B}_j=\big[-\frac{e^{j-1}\pi}{2\log |Q|}, -\frac{e^{j-2}\pi}{2\log |Q|}\big]
  \cup \big[\frac{e^{j-2}\pi}{2\log |Q|}, \frac{e^{j-1}\pi}{2\log |Q|}\big]$ for $2\leq j< \lfloor \log \log |Q|\rfloor :=K$. We further denote
  $\mathcal{B}_K=[-\pi/2,\pi/2]\setminus \bigcup_{1\leq j<K} \mathcal{B}_j$.

Observe that for any $t_1\in [0, \pi/2]$,  we have $\mathcal{D}\subset [0,\pi/2] \subset t_1+[-\pi/2,\pi/2]\subset \bigcup_{1\leq j\leq K} t_1+\mathcal{B}_j$. Thus
if we denote $\mathcal{A}_j=\mathcal{B}_j\cap (-t_1+\mathcal{D})$, then $(t_1+\mathcal{A}_j)_{1\leq j\leq K}$ form a partition of $\mathcal{D}$.
We apply Hölder's inequality twice to deduce that
\begin{align}
\label{LintoverD}
\begin{split}
    & \bigg(\int_{\mathcal{D}}|\mathcal{L}(e^{iv}/\sqrt{q},\chi)|dv\bigg)^{2m-k} \\
      \leq & \bigg( \sum_{1\leq j\leq K} \frac{1}{j}\cdot  j \int_{t_1+\mathcal{A}_j} |\mathcal{L}(e^{iv}/\sqrt{q},\chi)|dv \bigg)^{2m-k} \\
     \leq & \bigg(\sum_{1\leq j\leq K} j^{2m-k} \bigg( \int_{t_1+\mathcal{A}_j} \big|\mathcal{L}(e^{iv}/\sqrt{q},\chi)\big|dv  \bigg)^{2m-k}\bigg)
     \bigg(\sum_{1\leq j\leq K } j^{-(2m-k)/(2m-k-1)} \bigg)^{2m-k-1} \\
     \ll & \sum_{1\leq j\leq K} j^{2m-k} \bigg( \int_{t_1+\mathcal{A}_j} \big|\mathcal{L}(e^{iv}/\sqrt{q},\chi)\big|dv \bigg)^{2m-k} \\
     \leq & \sum_{1\leq j\leq K} j^{2m-k} |\mathcal{B}_j|^{2m-k-1} \int_{t_1+\mathcal{A}_j} \big|\mathcal{L}(e^{iv}/\sqrt{q},\chi)\big|^{2m-k}dv.
\end{split}
\end{align}
  We denote for $\mathbf{t}=(t_1,\ldots,t_k)$,
$$\mathcal{L}(\mathbf{t},v)=\sum_{\chi \in X_Q^*}\prod_{a=1}^k|\mathcal{L}(e^{it_a}/\sqrt{q},\chi)| \cdot |\mathcal{L}(e^{iv}/\sqrt{q},\chi)|^{2m-k}.$$
  We then deduce from \eqref{Lintdecomp} and \eqref{LintoverD} that
\begin{align}
\label{Intbound}
\begin{split}
  \sum_{\chi \in X_Q^*}\Big ( \int_{0}^{\pi/2} \Big |\mathcal{L}(e^{it}/\sqrt{q},\chi)  \Big |dt\Big )^{2m} \ll &
    \sum_{1\leq l_0\leq K} l_0^{2m-k} |\mathcal{B}_{l_0}|^{2m-k-1} \int_{[0,\pi/2]^k}\int_{t_1+\mathcal{A}_{l_0}}\mathcal{L}(\mathbf{t},v)dv d\mathbf{t}  \\
     \ll &  \sum_{1\leq l_0, l_1, \ldots l_{k-1}\leq K} l_0^{2m-k} |\mathcal{B}_{l_0}|^{2m-k-1} \int_{\mathcal{C}_{l_0,l_1, \cdots, l_{k-1}}} \mathcal{L}(\mathbf{t},v)dv d\mathbf{t},
\end{split}
\end{align}
where
$$\mathcal{C}_{l_0,l_1, \cdots, l_{k-1}}=\{(t_1,\ldots,t_k,v)\in [0,\pi/2]^{k+1}: v\in t_1+ \mathcal{A}_{l_0},\, |t_{i+1}-v|-|t_i-v|\in \mathcal{B}_{l_i},\ 1 \leq i \leq k-1\}.$$

  Note that the volume of the region $\mathcal{C}_{l_0,l_1, \cdots, l_{k-1}}$ is $\ll  \frac{e^{l_0+l_1+\cdots+l_{k-1}} }{(\log |Q|)^k}$. Also,
by the definition of $\mathcal{C}_{l_0,l_1, \cdots, l_{k-1}}$ we have $|t_1-v| \sim \frac{e^{l_0}}{\log |Q|}$ so that $w(|t_1- v|)\ll \frac{\log |Q|}{e^{l_0}}$, where we define for simplicity that $w(t)=\min (\log |Q|, 1/\overline{|t|})$. We deduce from the definition of $\mathcal{A}_j$ that $|t_2-v|\geq |t_1-v|$, so that $
\pi \geq  |t_2-v|= |t_1-v|+(|t_2-v|-|t_1-v|)\gg \frac{e^{l_0}}{\log |Q|}+\frac{e^{l_1}}{\log |Q|}$, which implies that $w(|t_2- v|)\ll \frac{\log |Q|}{e^{\max(l_0,l_1) }}$.
Similarly, we have  $w(|t_i- v|)\ll \frac{\log |Q|}{e^{\max(l_0,l_1,\ldots, l_{i-1}) }}$ for any $1 \leq i \leq k$.
Moreover, we have $\sum^{j-1}_{s=i}(|t_{s+1}-v|-|t_s-v|) \leq |t_j-t_i| $ for any $1 \leq i < j \leq k$, so that we have $w(|t_{j}- t_i|)\ll \frac{\log |Q|}{e^{\max(l_i,\ldots, l_{j-1} ) } }$. We then deduce from Corollary \ref{cor1} that for $(t_1,\ldots,t_k,v)\in \mathcal{C}_{l_0,l_1, \cdots, l_{k-1}}$,
\begin{align*}
     & \mathcal{L}(\mathbf{t},v) \\
     \ll & \varphi(Q)(\log |Q|)^{((2m-k)^2+k)/4}
     \bigg(\prod^{k-1}_{i=0}\frac{\log |Q|}{e^{ \max(l_0,l_1,\ldots, l_{i}) }} \bigg)^{(2m-k)/2}
     \bigg(\prod^{k-1}_{i=1} \prod^{k}_{j=i+1}\frac{\log |Q|}{e^{\max(l_i,\ldots, l_{j-1} ) } } \bigg)^{1/2} \\
     = & \varphi(Q)(\log |Q|)^{m^2}  \exp\Big( -\frac {2m-k}{2}\sum^{k-1}_{i=0}\max(l_0,l_1,\ldots, l_{i})-\frac 12\sum^{k-1}_{i=1} \sum^{k}_{j=i+1}\max(l_i,\ldots, l_{j-1} )\Big).
\end{align*}

 Moreover,  we have $|\mathcal{B}_{l_0}|\ll \frac{e^{l_0}}{\log |Q|}$, so that we have
\begin{align}
\label{firstcase}
\begin{split}
       &  \sum_{\substack{1\leq l_0 \leq K \\ 1\leq l_1, \ldots l_{k-1}\leq K}}  l_0^{2m-k} |\mathcal{B}_{l_0}|^{2m-k-1} \int_{\mathcal{C}_{l_0,l_1, \cdots, l_{k-1}}} \mathcal{L}(\mathbf{t},v)dvd\mathbf{t} \\
    \ll & \varphi(Q)(\log |Q|)^{(m-1)^2}   \\
    &  \times \sum_{\substack{1\leq l_0 \leq K \\ 1\leq l_1, \ldots l_{k-1}\leq K}}  l_0^{2m-k}\exp\Big( (2m-k-1)l_0+\sum^{k-1}_{i=0}l_i-\frac {2m-k}{2}\sum^{k-1}_{i=0}\max(l_0,l_1,\ldots, l_{i})-\frac 12\sum^{k-1}_{i=1} \sum^{k}_{j=i+1}\max(l_i,\ldots, l_{j-1} )\Big) \\
    = & \varphi(Q)(\log |Q|)^{(m-1)^2}   \\
    &  \times \sum_{\substack{1\leq l_0 \leq K \\ 1\leq l_1, \ldots l_{k-1}\leq K}}  l_0^{2m-k}\exp\Big( \frac {2m-k}{2}l_0+\frac 12\sum^{k-1}_{i=1}l_i-\frac {2m-k}{2}\sum^{k-1}_{i=1}\max(l_0,l_1,\ldots, l_{i})-\frac 12\sum^{k-1}_{i=1} \sum^{k}_{j=i+2}\max(l_i,\ldots, l_{j-1} )\Big).
\end{split}
\end{align}
   
     We now set $k=3$ to see that in this case we have
\begin{align*}
\begin{split}
     &  \frac {2m-k}{2}l_0+\frac 12\sum^{k-1}_{i=1}l_i-\frac {2m-k}{2}\sum^{k-1}_{i=1}\max(l_0,l_1,\ldots, l_{i})-
       \frac 12\sum^{k-1}_{i=1} \sum^{k}_{j=i+2}\max(l_i,\ldots, l_{j-1} ) \\
       \leq &  -(2m-4)\max(l_0,\ldots, l_{k-1} ). 
\end{split}
\end{align*}   
   
     We deduce from \eqref{Intbound}, \eqref{firstcase} and the above that 
\begin{align}
\label{firstcasesimplified}
\begin{split}
 &  \sum_{\chi \in X_Q^*}\Big ( \int_{0}^{\pi/2} \Big |\mathcal{L}(e^{it}/\sqrt{q},\chi)  \Big |dt\Big )^{2m}    \\
    \ll &  \sum_{\substack{1\leq l_0 \leq K \\ 1\leq l_1, \ldots l_{k-1}\leq K}}  l_0^{2m-k} |\mathcal{B}_{l_0}|^{2m-k-1} \int_{\mathcal{C}_{l_0,l_1, \cdots, l_{k-1}}} \mathcal{L}(\mathbf{t},v)dv d\mathbf{t} \\
     \ll &   \varphi(Q)(\log |Q|)^{(m-1)^2} \sum_{\substack{1\leq l_0<K \\ 1\leq l_1, \ldots l_{k-1}\leq K}}  l_0^{2m-k}
     \exp\Big( -(2m-4)\max(l_0,\ldots, l_{k-1} )\Big)  \\
    \ll &  \varphi(Q)(\log |Q|)^{(m-1)^2},
\end{split}
\end{align}   
      where the last estimation above follows by noting that we have $m>2$. 
   
   Note that Corollary \ref{cor1} is still valid with $\theta_i$ being replaced by $\pi-\theta_i$ on the left-hand side of \eqref{mathcalLestimation} while keeping
  $\theta_j, \theta_l$ intact on the right-hand side of \eqref{mathcalLestimation}. Using this, one checks that our arguments above carry over to show that 
\begin{align}
\label{Intbound1}
\begin{split}
    &  \sum_{\chi \in X_Q^*}\Big ( \int_{0}^{\pi/2} \Big |\mathcal{L}(-e^{it}/\sqrt{q},\chi) \Big | dt\Big )^{2m} \ll   \varphi(Q)(\log |Q|)^{(m-1)^2}. 
\end{split}
\end{align}     
   We then deduce from \eqref{finiteintest0}, \eqref{finiteintest1}, \eqref{firstcasesimplified} and \eqref{Intbound1} that the estimation given in 
   \eqref{finiteintest} holds. This completes the proof of the proposition.
\end{proof}

\vspace*{.5cm}

\noindent{\bf Acknowledgments.}  S.B. would like to thank Beihang University in Beijing for its great hospitality during his visit in May 2024, where part of this work was started. P. G. is supported in part by NSFC grant 11871082. The author would like to thank the
anonymous referee for his/her careful inspection of the paper and many valuable suggestions.

\bibliography{biblio}
\bibliographystyle{amsxport}

\end{document}